\documentclass{amsart}
\usepackage{amssymb, amsmath}
\title{On the equationally Artinian groups}
\author{M. Shahryari, J. Tayyebi}

\address{M. Shahryari: Department of Pure Mathematics,  Faculty of Mathematical
Sciences, University of Tabriz, Tabriz, Iran}
\email{mshahryari@tabrizu.ac.ir}
\address{j. Tayyebi: Department of Pure Mathematics,  Faculty of Mathematical
Sciences, University of Tabriz, Tabriz, Iran}
\email{j.tayyebi2017@gmail.com}

\newcommand{\Rad}{\mathrm{Rad}}
\newcommand{\VH}{V_H(S)}
\newcommand{\GX}{G[X]}

\newtheorem{corollary}{Corollary}

\newtheorem {theorem}{Theorem}
\newtheorem{lemma}{Lemma}

\begin{document}

\maketitle
\begin{abstract}
In this article, we study the property of being equationally Artinian in groups. We prove that a finite extension of an equationally Artinian group is again equationally Artinian. We also show that a quotient of an equationally Artinian group of the form $G[t]$ by a normal subgroup which is a finite union of radicals, is again equationally Artnian. This provides a large class of examples of equationally Artinian groups.
\end{abstract}


\vspace{1cm}
In the mid twentieth century, Alfred Tarski asked whether  two arbitrary non-abelian free groups are elementary equivalent. To answer this question, it was necessary to investigate systems of equations over groups. Makanin and Razborov proved that the existence of solutions for systems of equations over free groups is a decidable problem and an algorithm to solve such systems of equation is discovered (Makanin-Razborov diagrams, see \cite{Mak} and \cite{Raz}). The work of Makanin and Razborov as well as many other mathematicians was the beginning of {\em algebraic geometry over groups}. Since then, this new area of algebra was the subject of important studies in group theory. The work of Baumslag, Myasnikov and Remeslennikov provides a complete account of this new subject, \cite{BMR}. Positive solution to the problem  of Tarski is discovered by Kharlampovich, Myasnikov and Sela at the the beginning of the recent century (see \cite{KMTarski}, \cite{KM}, \cite{KMTarski2} and \cite{SEL}). After that, many mathematicians investigated the algebraic geometry over general algebraic systems and this new area of algebra is now known as {\em universal algebraic geometry}. The reader can see the works of Daniyarova, Myasnikov, and Remeslennikov as well as the lecture notes of Plotkin as introduction to this branch, \cite{DMR1}, \cite{DMR2}, \cite{DMR3}, \cite{DMR4}, and \cite{Plot}.

One of the very important notions in the algebraic geometry of groups (as well as other algebraic structures) is the property of being {\em equationally Noetherian}. Note that if $S$ is a system of equations over a group $A$, then we say that the system $S$ implies an equation $w\approx 1$, if every solution of $S$ in $A$ is also a solution of $w\approx 1$. This gives us an equational logic over the group $A$ which is not in general similar to the first order logic. For example, the compactness theorem may fails in this equational logic. There are examples of groups such that the compactness for the systems of equations fails (see \cite{BMR} and \cite{DMR3} for some examples). In some groups, every system of equations is equivalent to a finite subsystem, such groups are called equationally Noetherian. Free groups, Abelian groups, linear groups over Noetherian rings and torsion-free hyperbolic groups are equationally Noetherian. To see interesting properties of this types of groups, the reader can consult \cite{ModSH} and \cite{ShahShev}. This kind of groups have very important roles in algebraic geometry of groups. There are many equivalent conditions for the property of being equationally Noetherian, for example, it is know that a group $A$ has this property, if and only if, for any natural number $n$, every descending chain of algebraic sets in $A^n$ is finite. According to this equivalent condition, in \cite{ModSH} and \cite{ModSH2}, the dual property of being {\em equationally Artinian}  is defined. A group $A$ is equationally Artinian, if and only if, for any natural number $n$, every ascending  chain of algebraic sets in $A^n$ is finite. In \cite{ModSH2}, many equivalent conditions to this property is given.

In 1997, Baumslag, Myasnikov, and   Romankov proved two important theorems about equationally Noetherian groups: first, they showed that a virtually equationally Noetherian group is equationally Noetherian. They also showed that quotient of an equationally Noetherian group by a normal subgroup which is a finite union of algebraic sets, is again equationally Noetherian (see \cite{BMRom}). In this Article we prove similar results for the case of equationally Artinian groups. These results will provide  a large class of examples for equationally Artinian groups.

\section{Preliminaries}
Let $G$ be an arbitrary group and suppose that $X=\{x_1, \ldots, x_n\}$ is a finite set of variables. Consider the free product $\GX=G\ast F[X]$, where $F[X]$ is the free group over $X$. Every element $w\in \GX$ corresponds to an equation $w\approx 1$, which is called a group equation with coefficients from $G$. If $w=w(x_1,\ldots, x_n, g_1, \ldots, g_m)\in G[X]$, then the expression $w\approx 1$ is a $G$-equation with  coefficient $g_1,\ldots, g_m\in G$. Suppose $H$ is a group which contains $G$ as a distinguished subgroup. Then we say that $H$ is a $G$-group. A tuple $\overline{h}=(h_1, \ldots, h_n)\in H^n$ is called a root of the equation $w\approx 1$, if
$$
w(h_1, \ldots, h_n, g_1, \ldots, g_m)=1.
$$
An arbitrary set of $G$-equations is called {\em a system of equation with coefficients from $G$}. The set of all common roots of the elements of $S$ in $H$ is called the corresponding {\em algebraic set} of $S$ and denoted by $\VH$. Clearly, the intersection of a non-empty family  of algebraic sets is again an algebraic set but the same is not true for unions of algebraic sets. If we define a closed subset of $H^n$ to be an arbitrary intersection of finite unions of algebraic sets, then we get a topology on $H^n$, which is known as {\em Zariski topology}.

For a subset $E\subseteq H^n$, we define the corresponding radical $\Rad(E)$ to be the set of all elements $w\in \GX$ such that every element of $E$ is a solution of $w\approx 1$. This is a normal subgroup of $\GX$ which is called the {\em radical} of $E$ and the quotient group $\Gamma(E)=\GX/\Rad(E)$ is called the {\em coordinate group} of $E$. Similarly, for a system $S$, we define its radical to be $\Rad_H(S)=\Rad(\VH)$. This is the largest system of $G$-equations equivalent to $S$ over $H$. The corresponding coordinate group is $\Gamma_H(S)=\GX/\Rad_H(S)$. It is proved that the study of coordinate groups is equivalent to the study of Zariski topology, i.e. algebraic geometry of $H$ reduces to the study coordinate groups, \cite{BMR}.

A $G$-group $H$ is called $G$-equationally Noetherian, if for every system $S$, there exists a finite subsystem $S_0$, such that $\VH=V_H(S_0)$. Such $G$-groups have important role in the study of algebraic geometry over $G$-groups. There are two extremal cases: if $G=1$, we say that $H$ is $1$-equationally Noetherian or equationally Noetherian without coefficients, and if $G=H$, then we say that $H$ is equationally Noetherian (or equationally Noetherian in Diophantine sense). It is proved that a $1$-equationally Noetherian finitely generated group is equationally Noetherian as well, \cite{BMR}. The class of equationally Noetherian groups is very large, containing all  Free groups, Abelian groups, linear groups over Noetherian rings and torsion-free hyperbolic groups are equationally Noetherian. It is not hard to see that the following statements are equivalent for a $G$-group $H$:\\

i- $H$ is equationally Noetherian.

ii- the Zariski topology on $H^n$ is Noetherian for all $n$.

iii- every chain of coordinate groups and epimorphisms
$$
\Gamma(E_1)\to \Gamma(E_2)\to \Gamma(E_3)\to \cdots
$$
is finite. \\

The authors of \cite{BMRom} proved two important theorems about equationally Noetherian groups. The first theorem shows that a finite extension of an equationally Noetherian group is again equationally Noetherian. The second theorem says that if $G$ is equationally Noetherian and $N$ is a normal subgroup which is a finite union of algebraic sets (in Diophantine case), then $G/N$ is also equationally Noetherian. In this article, we are dealing with the dual notion, the property of being equationally Artinian and we prove the similar statements for this type of groups.

\section{Equationally Artinian groups}
Equationally Artinian algebras introduced in \cite{ModSH} and \cite{ModSH2}. In this section, we review this notion for the case of $G$-groups. We say that a $G$-group $H$ is $G$-equationally Artinian, if for any $n$, every ascending chain of algebraic sets in $H^n$ terminates. This is not equivalent to the property of being Artinian for the Zariski topology, instead we define a new topological space which becomes Noetherian if $H$ is equationally Artinian. Suppose
$$
T=\{ u\Rad(E):\ E\subseteq H^n,\ u\in \GX\}.
$$
Since the intersection of two cosets of radicals, is again a coset of a radical subgroup, so this set $T$ gives us a topology on the set $\GX$ which is called {\em the radical topology} on $\GX$ corresponding to $H$ (this topology is finer than the previous one defined in \cite{ModSH2}, but every thing remains unchanged here). Every closed set in $\GX$ is an arbitrary intersection of finite unions of cosets of the form $u\Rad(E)$, with $E\subseteq H^n$ and  $u\in \GX$. In \cite{ModSH2}, it is proved that the following statements are equivalent for a $G$-group $H$:\\

i- $H$ is  $G$-equationally Artinian.

ii- for any $n$ and any subset $E\subseteq H^n$, there exists a finite subset $E_0\subseteq E$, such that $\Rad(E)=\Rad(E_0)$.

iii- the corresponding radical topology over $\GX$ is Noetherian.\\

By $(EA)_G$, we denote the class of all $G$-equationally Artinian $G$-groups, by $(EA)_1$, the class of $1$-equationally Artinian groups and $EA$ will be used for the class of Equationally Artinian groups (Diophantine case where $G=H$). In this article, we first prove the following theorem.

\begin{theorem}
Let $G$ be a finitely generated group and let $H$ be a $G$-group. If $H\in (EA)_1$, then $H\in (EA)_G$, and as a result, any finitely generated element of $(EA)_1$ is equationally Artinian.
\end{theorem}

As an application, we will prove that every finitely generated Abelian group is equationally Artinian. Note that, this is not true for all Abelian groups. It is proved that (see \cite{ModSH}) every Abelian group is equationally Artinian if and only if, every finitely generated torsion-free Abelian group satisfies the minimum condition, but clearly the infinite cyclic group does not satisfy the minimum condition, so there are Abelian groups which are not equationally Artinian. For example, we will show that the additive group $\mathbb{Q}/\mathbb{Z}$ is not so, but in fact there are many infinitely generated Abelian groups (for example, additive groups of the fields $\mathbb{Q}$ and $\mathbb{R}$) which are equationally Artinian. As a result, we see that the classification of equationally Artinian Abelian groups is an interesting problem.

Our next theorem, will concern the groups with a finite index equationally Artinian subgroup.

\begin{theorem}
Let a group $A$ contains a finite index subgroup $H$ which is equationally Artinian. Then $A$ is also equationally Artinian.
\end{theorem}

This theorem enables us to conclude that any virtually finitely generated Abelian group is equationally Artinian as well as any finite extension of the additive group of any field. This gives us a large class of examples of such groups. The proof of Theorem 2, requires some technics of \cite{BMRom} but it has also some new arguments which are not applied in \cite{BMRom}.

Note that the quotient of an equationally Artinian group is not necessarily equationally Artinian, but, there exists an important situation, the quotient in which, has this property. Our last theorem concerns with these situations. Note that in this theorem, we use the group $G[t]=G\ast \langle t\rangle$.

\begin{theorem}
Let $G$ be an arbitrary group such that $G[t]$ is equationally Artinian. Let  $R$ be a normal subgroup of $G[t]$ which is  a finite union of radical subgroups. Then $G[t]/R$ is also equationally Artinian.
\end{theorem}

\section{The proofs}
First, we prove Theorem 1.

\begin{proof}
Let $a_1, \ldots, a_k$ be a finite set of generators for the group $G$. Suppose $E\subseteq H^n$. We prove that there exists a finite subset $E_0\subseteq E$, such that $\Rad_G(E)=\Rad_G(E_0)$ (note that, here $\Rad_G$ denotes the radical with coefficient in $G$).   Let $S=\Rad_G(E)\subseteq \GX$. Every element of $S$ has the form
$$
w=w(x_1, \ldots, x_n, a_1, \ldots, a_k).
$$
We replace every coefficient $a_i$ by a new variable $y_i$, and then a coefficient-free system of equations $S(\overline{x}, \overline{y})$ appears. Let $T=E\times \{(a_1, \ldots, a_k)\}\subseteq H^{n+k}$. Now, since $H\in (EA)_1$, so there is a finite subset $T_0\subseteq T$, such that $\Rad_1(T)=\Rad_1(T_0)$. Clearly, we have $T_0=E_0\times \{(a_1, \ldots, a_k)\}$, for some finite subset $E_0\subseteq E$. Obviously, $S(\overline{x}, \overline{y})\subseteq \Rad_1(T)$. Let $u(\overline{x}, \overline{y})\in \Rad_1(T)$. Then for all $\overline{e}\in E$, we have $u(\overline{e}, \overline{a})=1$, so $u(\overline{x}, \overline{a})\in \Rad_G(E)$, and therefore $u\in S(\overline{x}, \overline{y})$. This proves that $S(\overline{x}, \overline{y})=\Rad_1(T)$, and hence $S(\overline{x}, \overline{y})=\Rad_1(T_0)$.

Now, we show that $S(\overline{x}, \overline{a})=\Rad_G(E_0)$. Suppose $w(\overline{x}, \overline{a})\in S(\overline{x}, \overline{a})$. For any $\overline{e}\in E_0$, we have $w(\overline{e}, \overline{a})=1$, so $w(\overline{x}, \overline{a})\in \Rad_G(E_0)$. Conversely, if $w(\overline{x}, \overline{a})\in \Rad_G(E_0)$, then for $w(\overline{x}, \overline{y})\in \Rad_1(T_0)=\Rad_1(T)$, and this shows that $w(\overline{x}, \overline{a})\in \Rad_G(E)=S(\overline{x}, \overline{a})$. This proves that $\Rad_G(E)=\Rad_G(E_0)$ and hence $H\in (EA)_G$.
\end{proof}

Theorem 1, enables us to prove that every finitely generated Abelian group belongs to the class $EA$. To do this, we prove that the infinite cyclic group  is equationally Artinian.

\begin{lemma}
Let $H=\langle a\rangle$ be infinite cyclic group. Then $H$ is equationally Artinian.
\end{lemma}

\begin{proof}
We first show that $H\in (EA)_1$. Let $E\subseteq H^n$. Every element of $E$ has the form $\overline{e}=(a^{j_1}, \ldots, a^{j_n})$ for some integers $j_1, \ldots, j_n$. Let $w=x_1^{\alpha_1}x_2^{\alpha_2}\ldots x_n^{\alpha_n}\in \Rad_1(E)$. Then $w(\overline{e})=1$ and hence $a^{j_1\alpha_1+\cdots+j_n\alpha_n}=1$. This shows that
$$
\Rad_1(E)=\bigcap_{j_1, \ldots, j_n}\{x_1^{\alpha_1}x_2^{\alpha_2}\ldots x_n^{\alpha_n}: (a^{j_1}, \ldots, a^{j_n})\in E, j_1\alpha_1+\cdots+j_n\alpha_n=0\}.
$$
Suppose
$$
E=\{ (a^{j_1^{(1)}}, \ldots, a^{j_n^{(1)}}), (a^{j_1^{(2)}}, \ldots, a^{j_n^{(2)}}), (a^{j_1^{(3)}}, \ldots, a^{j_n^{(3)}}), \ldots\}.
$$
Suppose $S$ is the following set of equations
$$
j_1^{(t)}\alpha_1+\cdots+j_n^{(t)}\alpha_n=0,\ \ (t=1, 2, 3, \ldots).
$$
Since the additive group $\mathbb{Z}$ is equationally Noetherian, so there exists a finite subset $S_0\subseteq S$, such that $V_{\mathbb{Z}}(S)=V_{\mathbb{Z}}(S_0)$. Suppose $S_0$ consists of the equations
$$
j_1^{(t)}\alpha_1+\cdots+j_n^{(t)}\alpha_n=0\ \ (t=1, 2, \ldots, m).
$$
Let $E_0=\{ (a^{j_1^{(1)}}, \ldots, a^{j_n^{(1)}}), (a^{j_1^{(2)}}, \ldots, a^{j_n^{(2)}}),\ldots,  (a^{j_1^{(m)}}, \ldots, a^{j_n^{(m)}})\}$. Then we have obviously, $\Rad_1(E)=\Rad_1(E_0)$. This shows that $H$ is $1$-equationally Artinian and hence by Theorem 1, it belongs to $EA$.
\end{proof}

Now, we show that any direct product of finitely many element of $(EA)_1$ is again in $(EA)_1$. This will prove  that every finitely generated Abelian group belongs to $(EA)_1$ and hence to $EA$.

\begin{lemma}
Suppose $A$ and $B$ are equationally Artinian. Then so is $A\times B$.
\end{lemma}

\begin{proof}
For  a number $n$ and a subset $E\subseteq (A\times B)^n$, suppose that
$$
E=\{ c_i=(u_1^i, u_2^i, \ldots, u_n^i):\ i\in I\},
$$
where $I$ is an index set. We have $u_i^j=(a_i^j, b_i^j)$, for some $a_i^j\in A$ and $b_i^j\in B$. Now, let
$$
T=\{ t_i=(a_1^i, a_2^i, \ldots, a_n^i):\ i\in I\},
$$
and
$$
S=\{ s_i=(b_1^i, b_2^i, \ldots, b_n^i):\ i\in I\}.
$$
Since $A$ and $B$ are equationally Artinian, so there are two finite subsets $T_0\subseteq T$ and $S_0\subseteq S$, such that
$$
\Rad_A(T)=\Rad_A(T_0), \ \Rad_B(S)=\Rad_B(S_0).
$$
Suppose for example $T_0=\{ t_1, \ldots, t_l\}$ and $S_0=\{ s_1, \ldots, s_k\}$ and $k\geq l$. Suppose $t_i=(a_1^i, \ldots, a_n^i)$ and $s_i=(b_1^i, \ldots, b_n^i)$. Using these elements, we can define a  finite subset
$$
E_0=\{ c_i=((a_1^i, b_1^i), \ldots, (a_n^i, b_n^i): 1\leq i\leq l\},
$$
such that $\Rad(E)=\Rad(E_0)$. This shows that $A\times B$ is equationally Artinian.
\end{proof}

Summarizing, we have

\begin{corollary}
Every finitely generated Abelian group is equationally Artinian.
\end{corollary}

There are also infinitely generated Abelian groups which are equationally Artinian: let $K$ be a field and consider its additive group $H=(K, +)$. Every equation with coefficient in $K$ has the form $a_1x_1+\cdots+a_nx_n=b$ for some elements $a_1, \ldots, a_n\in \mathbb{Z}, b\in K$, so the corresponding algebraic set is an affine subspace of $K^n$. This shows that every ascending chain of algebraic sets terminates and hence $H$ is equationally Artinian. However, some Abelian groups are not equationally Artinian. For example, consider the additive group $H=\mathbb{Q}/\mathbb{Z}$. Let
$$
E=\{ \frac{1}{p}+\mathbb{Z}:\ p=prime\}\subseteq H^1.
$$
If $w(x)=mx+(\frac{a}{b}+\mathbb{Z})\in \Rad(E)$, then for any prime $p$,  we have $w(\frac{1}{p}+\mathbb{Z})=\mathbb{Z}$, and this means that for any prime $p$, $\frac{m}{p}+\frac{a}{b}\in \mathbb{Z}$, which is not true.

Before proving Theorem 2, we introduce some notations from \cite{BMRom}. Let a group $A$ be the semidirect product of a finite subgroup $T$ and a normal subgroup $H$. Assume that $T=\{ t_1=1, t_2, \ldots, t_k\}$. Let $w(x_1, \ldots, x_n, g_1, \ldots, g_m)$ be a group word with coefficients  in $A$ and $v\in A^n$. We can express $v$ uniquely in the form $v=(s_1h_1,\ldots, s_nh_n)$ with $s_i\in T$ and $h_i\in H$. We also have $g_i=r_ib_i$ for unique elements $r_i\in T$ and $b_i\in H$. Define the map $\lambda:A^n\to T^n$ by $\lambda(v)=(s_1, \ldots, s_n)$ and
$$
\overline{w}(x_1, \ldots, x_n)=w(x_1, \ldots, x_n, r_1, \ldots, r_m).
$$
Note that $\overline{w}$ is an element of $T[X]$ which depends only on $w$. For any $1\leq i\leq n$ and $1\leq j\leq k$, define $h_{ij}=t_j^{-1}h_it_j\in H$. Denote the tuple
$$
(h_{11}, \ldots, h_{1k}, \ldots, h_{n1}, \ldots, h_{nk})
$$
by $v^{\prime}$. Consider the new variables $y_{ij}$ for $1\leq i\leq n$ and $1\leq j\leq k$. In \cite{BMRom}, it is proved that there exists a unique element
$$
w^{\prime}_v\in H[y_{11} \ldots, y_{1k}, \ldots, y_{n1}, \ldots, y_{nk}],
$$
such that $w(v)=\overline{w}(\lambda(v))w^{\prime}_v(v^{\prime})$, and $w^{\prime}_v$ depends only on the value of $\lambda(v)$. As a result, it is shown that $v\in A^n$ is a root of $w\approx 1$, if and only if, $\lambda(v)$ is a root of $\overline{w}\approx 1$ and $v^{\prime}$ is a root of $w^{\prime}_v\approx 1$. We are now ready to prove Theorem 2.

\begin{proof}
Replacing $H$ by its core, we can suppose that $H$ is a normal subgroup of $A$ with finite index. Let $T=A/H$. Then $A$ embeds into the wreath product $H\wr T$. Recall that this wreath product is the semidirect product of $T$ and $H^{|T|}$. We know that (Lemma 2), $H^{|T|}$ is equationally Artinian and any subgroup of an equationally Artinian group is again equationally Artinian. So, it is enough to prove our theorem using the further assumption $A=TH$, with $T$ finite, $H$ normal and $T\cap H=1$. We will use all the above notations.

Suppose $E\subseteq A^n$ is an algebraic set and $S=\Rad_A(E)$. We must show that there exists a finite subset $E_0\subseteq E$, such that $\Rad_A(E_0)=S$. Let $\overline{S}=\{ \overline{w}:\ w\in S\}$ (see the above discussion). Suppose
$$
V_T(\overline{S})=\{ v_1, \ldots, v_d\}\subseteq T^n.
$$
For any $1\leq i\leq d$, put $L_i=V_H(S^{\prime}_{v_i})\subseteq H^{nk}$. Here $S^{\prime}_{v_i}$ denotes the set of all $w^{\prime}_{v_i}$, such that $w\in S$. Define also
$$
K_i=\{ \overline{h}\in H^n:\ (\overline{h})^{\prime}\in L_i\}\subseteq H^n.
$$
We have $(K_i)^{\prime}\subseteq H^{nk}$ and since $H$ is equationally Artinian, there exists a finite subset $K^0_i\subseteq K_i$, such that
$$
\Rad_H((K^0_i)^{\prime})=\Rad_H((K_i)^{\prime}).
$$
Assume that $E_0=\cup_{i=1}^dv_iK^0_i\subseteq A^n$. We show that $E_0\subseteq E$. Let $v_i\overline{h}\in E_0$. Then $\overline{h}\in K_i$ and hence
$$
\overline{S}(\lambda(v_i\overline{h}))=\overline{S}(v_i)=1,
$$
and
$$
S^{\prime}_{v_i}((v_i\overline{h})^{\prime})\in S^{\prime}_{v_i}(L_i)=1.
$$
This means that $v_i\overline{h}\in V_A(S)=E$. Therefore $E_0\subseteq E$.

Now, we claim that $\Rad_A(v_iK^0_i)=\Rad_A(v_iK_i)$. To prove this claim, assume that $w$ belongs to the left hand side. Then $w(v_iK^0_i)=1$ and hence $w^{\prime}((v_iK^0_i)^{\prime})=1$. This shows that $w^{\prime}_{v_i}\in \Rad_H((v_iK^0_i)^{\prime})$. Recall that, by the definition of the map $v\mapsto v^{\prime}$, we have  $(v_iK^0_i)^{\prime})=(K^0_i)^{\prime}$ and hence $w^{\prime}_{v_i}\in\Rad_H((K^0_i)^{\prime})=\Rad_H((K_i)^{\prime})=\Rad_H((v_iK_i)^{\prime})$. Therefore, for any $\overline{h}\in K_i$, we have $w^{\prime}_{v_i}((v_i\overline{h})^{\prime})=1$, and since in the same time $\overline{w}(\lambda(v_i\overline{h}))=1$, we have $w(v_iK_i)=1$. This proves the claim.

We  now, prove that $\Rad_A(E_0)=\Rad_A(E)$. Let $w$ be an element of the left hand side and $v\in E$. We have $S(v)=1$ and
$$
w\in \bigcap_{i=1}^d\Rad_A(v_iK^0_i).
$$
Note that $v=\lambda(v)\overline{h}$, for some $\overline{h}\in H^n$. We have $\overline{S}(\lambda(v))=1$, so there is an index $i$ such that $\lambda(v)=v_i$. Therefore, $v=v_i\overline{h}$. On the other side, since $S^{\prime}_{v}(V^{\prime})=1$, so
$$
1=S^{\prime}_v(v^{\prime})=S^{\prime}_{v_i}((v_i\overline{h})^{\prime}).
$$
Hence, $(v_i\overline{h})^{\prime}\in L_i$, and therefore $\overline{h}\in K_i$. Now, by the above claim, we have
$$
w\in \Rad_A(v_iK^0_i)=\Rad_A(v_iK_i),
$$
and hence $w(v)=1$. This shows that $w\in \Rad_A(E)$.
\end{proof}

Theorem 2 shows that any virtually finitely generated Abelian group is equationally Artinian as well as any finite extension of the additive group of any field. This gives us a large class of examples of such groups. We now come to Theorem 3. Note that the similar theorem (\cite{BMRom}) for the equationally Noetherian case deals with the Zarizki topology of $G^1$ and its closed normal subgroups. The dual case here deals with the radical topology of $G[t]$ and its closed normal subgroups.

\begin{proof}
Assume that
$$
R=\bigcup_{i=1}^m\Rad_G(K_i),
$$
where $K_i\subseteq G$. Note that $G$ is equationally Artinian as $G[t]$ is so. Hence every $K_i$ can be chosen finite. Let $H=G[t]/R$ be not equationally Artinian. Hence there exists a number $n$ and a subset $E\in H^n$ such that $\Rad_H(E)\neq \Rad_H(E_0)$, for any finite subset $E_0\subseteq E$. Assume that $e_0\in E$ is an arbitrary element. As $\Rad_H(E)\neq \Rad_H(\{ e_0\})$, there exist elements $f_1\in \Rad_H(\{ e_0\})$ and $e_1\in E$, such that $f_1(e_1)\neq 1$. Similarly, we have $\Rad_H(E)\neq \Rad_H(\{ e_0, e_1\})$, so there exist elements $f_2\in \Rad_H(\{ e_0, e_1\})$ and $e_2\in E$, such that $f_2(e_2)\neq 1$. Repeating this argument, we obtain two infinite sequences
$$
f_1, f_2, f_3, \ldots \in H[X],
$$
$$
e_0, e_1, e_2, \ldots \in E,
$$
such that for any $i$, $f_i(e_0)=f_i(e_1)=\cdots =f_i(e_{i-1})=1$, but $f_i(e_i)\neq 1$. Note that, here $X=\{x_1, \ldots, x_n\}$ and so every element of $H[X]$ is a word in $t$ and elements of $X$ with coefficients in $G$. Suppose $q:G[t, X]\to H[X]$ is the canonical map sending elements of $G$ to their cosets, and fixing elements of $X$ and the element $t$. Suppose also that $\psi:(G[t])^n\to H^n$ is the map
$$
\psi(u_1, \ldots, u_n)=(u_1R, \ldots, u_nR).
$$
Choose a pre-image $\overline{f}_i$ for $f_i$ under $q$ and a pre-image $\overline{e}_i$ for $e_i$ under $\psi$. Hence, we have $\overline{f}_i\in G[t, X]$ and $\overline{e}_i\in (G[t])^n$. For any $i$, we have $f_i(e_0)=1$, so $\overline{f}_i(\overline{e}_0)\in R$. This shows that, there exists an infinite sequence of numbers
$$
i_1(0)<i_2(0)<i_3(0)<\cdots,
$$
and a number $1\leq p_0\leq m$, such that
$$
\overline{f}_{i_1(0)}(\overline{e}_0),  \overline{f}_{i_2(0)}(\overline{e}_0), \overline{f}_{i_3(0)}(\overline{e}_0), \ldots \in \Rad_G(K_{p_0}).
$$
Equivalently, this shows that for all $s$, we have
$$
\overline{f}_{i_s(0)}\in \Rad_{G[t]}(\overline{e}_0(K_{p_0})).
$$
By a similar argument, we obtain an infinite subsequence of $\{ i_s(0)\}$ of the form
$$
i_1(1)<i_2(1)<i_3(1)<\cdots,
$$
and a number $1\leq p_1\leq m$, such that for all $s$, we have
$$
\overline{f}_{i_s(1)}\in \Rad_{G[t]}(\overline{e}_1(K_{p_1})).
$$
We continue this process to find an infinite subsequence
$$
i_1(k)<i_2(k)<i_3(k)<\cdots,
$$
of the previous sequence, and a number $1\leq p_k\leq m$, such that
$$
\overline{f}_{i_s(k)}\in \Rad_{G[t]}(\overline{e}_k(K_{p_k})),
$$
for all $s$. Note that all sets $\overline{e}_i(K_{p_i})$ are finite as $K_i$'s are finite. Let
$$
K=\bigcup_{i=0}^{\infty}\overline{e}_i(K_{p_i})\subseteq (G[t])^n.
$$
By assumption, $G[t]$ is equationally Artinian, so there exists an index $l$, such that
$$
\Rad_{G[t]}(K)=\Rad_{G[t]}(\bigcup_{i=0}^{l}\overline{e}_i(K_{p_i}).
$$
Assume that $j>l$. Then for any $s$, we have
$$
\overline{f}_{i_s(j)}\in \bigcap_{i=1}^l\Rad_{G[t]}\overline{e}_i(K_{p_i})=\Rad_{G[t]}(K).
$$
Suppose $k=i_1(j)$. Then $\overline{f}_k\in \Rad_{G[t]}(\overline{e}_k(K_{p_k}))$, and hence $\overline{f}_k(\overline{e}_k)\in \Rad_G(K_{p_k})\subseteq R$. This shows that $f_k(e_k)=1$, a contradiction. Hence $H$ is equationally Artinian.
\end{proof}

As a final note we must say that Theorem 3 holds for the general case of $H=G[T]/R$, where $T$ is an arbitrary set of variables. At this moment we don't know if even the free group $F[T]$ is equationally Artinian. So, we are not sure about the situation of $G[T]$ as well.

\end{document}